\documentclass[10pt,a4paper,reqno]{amsart}

\usepackage{%
    amssymb%
    ,graphicx%
    ,mathrsfs%
    ,eucal%
    ,xcolor%
    ,enumitem%
    ,esint}
\usepackage[%
    pdfauthor={Piero D'Ancona}%
    ,colorlinks=true%
    ,linkcolor=magenta%
    ,citecolor=cyan%
    ,urlcolor=blue%
    ,hyperfootnotes=false%
]{hyperref}

\numberwithin{equation}{section}
\newtheorem{theorem}{Theorem}[section]
\newtheorem{corollary}[theorem]{Corollary}

\theoremstyle{remark}
\newtheorem{remark}{Remark}[section]

\theoremstyle{definition}
\newtheorem{definition}[theorem]{Definition}
\usepackage{stmaryrd}
  {\baselineskip0pt\normalfont\footnotesize
  \setlength\abovedisplayskip{1pt}
  \setlength\abovedisplayshortskip{1pt}
  \setlength\belowdisplayskip{1pt}
  \setlength\belowdisplayshortskip{1pt}
  \vskip2pt\par\noindent$\llbracket$\ignorespaces}%
  {\ignorespaces$\rrbracket$\vskip2pt}

\newcommand{\bra}[1]{\langle #1 \rangle}
\newcommand{\one}[1]{\mathbf{1}_{#1}}

\title[Kato smoothing]
{Kato smoothing and Strichartz estimates\\
for wave equations with magnetic potentials}
\date{\today}    
\author{Piero D'Ancona}
\address{Piero D'Ancona: Unversit\`a di Roma ``La Sapienza'',
Dipartimento di Matematica, Piazzale A.~Moro 2, I-00185 Roma, Italy}
\email{dancona@mat.uniroma1.it}
\keywords{}
\begin{document}
\begin{abstract}
  Let $H$ be a selfadjoint operator 
  and $A$ a closed operator
  on a Hilbert space $\mathcal{H}$.
  If $A$ is $H$-(super)smooth in the sense of Kato-Yajima,
  we prove that $AH^{-\frac14}$ is 
  $\sqrt{H}$-(super)smooth. This allows to include wave
  and Klein-Gordon equations 
  in the abstract theory at the same level of generality
  as Schr\"{o}dinger equations.

  We give a few applications and in particular, based on
  the resolvent estimates of Erdogan, Goldberg and Schlag
  \cite{ErdoganGoldbergSchlag09-a},
  we prove Strichartz estimates for wave equations 
  perturbed with
  large magnetic potentials on $\mathbb{R}^{n}$, $n\ge3$.
\end{abstract}
\maketitle



\section{Introduction}\label{sec:introduction}

In his fundamental 1965 paper \cite{Kato65-b}, Kato developes
a theory of similarity for small perturbations 
$H(\epsilon)=H+\epsilon V$
of an unbounded operator $H$ on a Hilbert space $\mathcal{H}$,
by constructing a bounded wave operator
$W(\epsilon)$
with the property $H(\epsilon)=W(\epsilon)H W^{-1}(\epsilon)$.
In the selfadjoint case $H=H^{*}$ the theory can be
precised and provides a bridge between
the dispersive properties of the Schr\"{o}dinger flow
$e^{itH}$ and uniform estimates for the resolvent
operator $R(z)=(H-z)^{-1}$. 

The relevance of Kato's theory
in the study of dispersive equations was understood
already in 
\cite{Yajima87-a} and
\cite{JourneSofferSogge91-a}.
A remarkable application was given in 
\cite{RodnianskiSchlag04-a} 
where Kato smoothing was used to give a simple
proof of Strichartz
estimates for the flow
$e^{it(-\Delta+V)}$ perturbed by a short range
potential $|V(x)|\lesssim \bra{x}^{-2-\epsilon}$.
The corresponding result for short range
magnetic potentials was proved in
\cite{DAnconaFanelli08-a} in the case of small potentials and
\cite{ErdoganGoldbergSchlag09-a} for large potentials
(for recent related results, see also
\cite{DAnconaFanelli06-a},
\cite{MarzuolaMetcalfeTataru08-a},
\cite{ErdoganGoldbergGreen13-a},
\cite{DAnconaSelberg12-a},
\cite{Green13-a}).

It is natural to investigate applications of Kato's theory
to the corresponding
wave-Klein-Gordon flow $e^{it \sqrt{H+\nu}}$ 
(with $H+\nu\ge0$).
The standard approach is a
reduction to the Schr\"{o}dinger flow $e^{itK}$ where
\begin{equation*}
  K=
  \begin{pmatrix}
   0 & 1 \\
   H &  0
  \end{pmatrix}
  \quad\implies \quad
  \exp(itK)=
    \begin{pmatrix}
     \cos(t \sqrt{H}) & \frac{i}{\sqrt{H}}\sin(t \sqrt{H}) \\
      i \sqrt{H}\sin(t \sqrt{H})&  \cos(t \sqrt{H})
    \end{pmatrix}
\end{equation*}
however this path leads to a loss
in the sharpness of the estimates, and in some cases
it requires
some ad-hoc argument to prove the necessary
resolvent estimates for $K$
(see e.g.~\cite{Mochizuki10-a}, \cite{DAnconaRacke12-a} or
\cite{BurqPlanchonStalker04-a}).

The first goal of this note is to deduce
smoothing estimates for abstract wave equations 
within the framework of Kato's theory:
given a non negative selfadjoint operator $H$
and a closed operator $A$ on the Hilbert space $\mathcal{H}$,
we prove that
\begin{center}
  $A$ is $H$-(super)smooth
  $\implies$
  $AH^{-\frac14}$ is $\sqrt{H}$-(super)smooth
\end{center}
in the sense of Kato-Yajima;
see Section \ref{sec:kato} for definitions and details and
in particular Theorem \ref{the:wave}.
Actually we prove that if $H+\nu\ge0$ for some $\nu\in \mathbb{R}$
and $A$ is $H$-(super)smooth, then
$A(H+\nu)^{-\frac14}$ is $\sqrt{H+\nu}$-(super)smooth.

It is clear that the range of applications is quite wide.
In Section \ref{sec:applications} we picked three.
The first two are mostly known results:
smoothing estimates for the flows
generated by powers of the Laplacian; and the
smoothing estimates for wave equations with
potentials of critical decay which were obtained in
\cite{BurqPlanchonStalker04-a}.

As a main application of the abstract result, we prove sharp
Strichartz estimates for wave equations perturbed with
large magnetic potentials, thus extending
the result for small potentials
in \cite{DAnconaFanelli08-a}.
This result is based on the resolvent estimate
due to Erdogan, Goldberg and Schlag
\cite{ErdoganGoldbergSchlag09-a}
for the magnetic Schr\"{o}dinger operator
\begin{equation*}
  H=(i \nabla+A(x))^{2}+V(x)
  \quad\text{ on $\mathbb{R}^{n}$, $n\ge3$}
\end{equation*}
under the following assumptions: $A(x)\in \mathbb{R}^{n}$,
$V(x)\in \mathbb{R}$; moreover for some $C,\epsilon_{0}>0$
\begin{equation*}
  |A(x)|+\bra{x}|V(x)|\le C \bra{x}^{-1-\epsilon_{0}},
\end{equation*}
\begin{equation*}
  \forall\  0<\epsilon<\epsilon_{0}, \quad
  \bra{x}^{1+\epsilon}A(x)\in \dot H^{\frac12}_{2n}(\mathbb{R}^{n})
  \cap
  C^{0}(\mathbb{R}^{n})
\end{equation*}
($\|u\|_{\dot H^{s}_{q}}:=\||D|^{s}u\|_{L^{q}}$,
$|D|^{s}:=(-\Delta)^{\frac s2}$);
finally, 0 is not an eigenvalue of $H$
nor a resonance, in the sense that
\begin{equation*}
  \bra{x}^{\frac{n-4}{2}-}u(x)\in L^{2}(\mathbb{R}^{n}),
  \qquad
  Hu=0
  \quad\implies\quad
  u \equiv0.
\end{equation*}
In dimension $n\ge5$ it is sufficient to assume that
0 is not an eigenvalue.
Under these assumptions, Erdogan, Goldberg and Schlag
proved that the Schr\"{o}dinger flow
$e^{itH}$ satisfies the same Strichartz estimates as the
free flow $e^{it \Delta}$. In Section 
\ref{sub:wave_equation_with_large} we combine their
resolvent estimate with
Theorem \ref{the:wavestr} to prove:

\begin{theorem}\label{the:waveintro}
  Assume $H$ is selfadjoint, nonnegative and satisfies the
  previous assumptions.
  Then the wave flow satisfies the
    non-endpoint Strichartz estimates
    \begin{equation*}
      \textstyle
      \||D|^{\frac1q-\frac1p}e^{it \sqrt{H}}f\|_{L^{p}L^{q}}
      \lesssim
      \|f\|_{\dot H^{\frac12}},
    \end{equation*}
    and
    \begin{equation*}
      \textstyle
      \||D|^{\frac1q-\frac1p}
         \sin(t \sqrt{H}) H^{-\frac12}g
         \|_{L^{p}L^{q}}
      \lesssim
      \|g\|_{\dot H^{-\frac12}},
    \end{equation*}
    for all $2<p\le \infty$, $2\le q<\frac{2(n-1)}{(n-3)}$
    with $2p^{-1}+(n-1)q^{-1}=(n-1)2^{-1}$.
\end{theorem}

\begin{remark}\label{rem:introendp}
  It is possible to prove the estimates also at the endpoint,
  using the ideas in \cite{IonescuKenig05-a},
  but this would lead us too far from the
  main goal of the paper.
\end{remark}

\begin{remark}\label{rem:introKG}
  By the same methods one can prove non-endpoint
  Strichartz estimates for the Klein-Gordon flow
  $e^{it \sqrt{H+\nu}}$
  perturbed with large potentials, provided $H+\nu\ge0$.
  We omit the details.
\end{remark}

\section{Abstract Kato smoothing}\label{sec:kato}

In this section we review
the basics of Kato's theory for the Schr\"{o}dinger
equation as developed
in \cite{Kato65-b} and \cite{KatoYajima89-a}
(see also \cite{ReedSimon78-a} and
\cite{Mochizuki10-a}), 
adding concise proofs when necessary,
and then we extend it to wave type equations.

Throughout this section $\mathcal{H}$, $\mathcal{H}_{1}$ are 
Hilbert spaces and $H$ is a selfadjoint operator on 
$\mathcal{H}$ with domain $D(H)$.
For $z\in \mathbb{C}\setminus \mathbb{R}$,
let $R(z)=(H-z)^{-1}$ be the resolvent operator of $H$, and
\begin{equation*}
  \Im R(z)=(2i)^{-1}(R(z)-R(\overline{z}))
\end{equation*} 
its imaginary part. 
A \emph{step function} $u(t):\mathbb{R}\to \mathcal{H}$
is a measurable function of bounded support
taking a finite number of values; measurability and
integrals of Hilbert-valued functions
are in the sense of Bochner. We also write for short
$L^{p}\mathcal{H}=L^{p}(\mathbb{R};\mathcal{H})$.
The following terminology was introduced in 
\cite{KatoYajima89-a}:

\begin{definition}\label{def:smoo}
  A closed operator $A$ from $\mathcal{H}$ to $\mathcal{H}_{1}$
  with dense domain $D(A)$ 
  is called:
  \\
  (i) $H$-\emph{smooth}, with constant $a$,
  if $\exists\epsilon_{0}$ such that for every 
  $\epsilon,\lambda\in \mathbb{R}$ with 
  $0<|\epsilon|< \epsilon_{0}$ the following
  uniform bound holds:
  \begin{equation}\label{eq:smoo}
    |(\Im R(\lambda+i \epsilon)A^{*}v,A^{*}v)_{\mathcal{H}_{1}}|
    \le
    a\|v\|_{\mathcal{H}_{1}}^{2},
    \qquad
    v\in D(A^{*});
  \end{equation}
  (ii) $H$-\emph{supersmooth}, with constant $a$, if 
  in place of \eqref{eq:smoo} one has
  \begin{equation}\label{eq:ssmoo}
    |(R(\lambda+i \epsilon)A^{*}v,A^{*}v)_{\mathcal{H}_{1}}|
    \le
    a\|v\|_{\mathcal{H}_{1}}^{2},
    \qquad
    v\in D(A^{*}).
  \end{equation}
\end{definition}

\begin{remark}\label{rem:normal}
  Note that \eqref{eq:smoo} implies that the range of
  $\Im R(z)A^{*}$ is contained in $D(A)$ and the
  (selfadjoint nonnegative) operator $A\Im R(z)A^{*}$ is a bounded
  operator on $\mathcal{H}_{1}$ 
  with norm equal to $a$. In a similar way,
  \eqref{eq:ssmoo} implies that the range of
  $R(z)A^{*}$ is contained in $D(A)$ and the
  operator $A R(z)A^{*}$ is a bounded
  operator on $\mathcal{H}_{1}$ 
  with norm not exceeding $2a$ (and not smaller than $a$). 
\end{remark}

\begin{theorem}\label{the:1}
  Let $A:\mathcal{H}\to \mathcal{H}_{1}$ be a closed operator
  with dense domain $D(A)$. 
  Then $A$ is $H$-smooth
  with constant $a$ if and only if, for any 
  $v\in \mathcal{H}$, one has
  $e^{-itH}v\in D(A)$ for almost every $t$
  and the following estimate holds:
  \begin{equation}\label{eq:smooest}
    \|Ae^{-itH}v\|_{L^{2}\mathcal{H}_{1}}
    \le
    2a^{\frac12}\|v\|_{\mathcal{H}}.
  \end{equation}
\end{theorem}

\begin{proof}
  This is proved in Lemma 3.6 and Theorem 5.1 of
  \cite{Kato65-b} (see also Theorem XIII.25 in
  \cite{ReedSimon78-a}).
\end{proof}

Thus $H$-smoothness is equivalent to the smoothing
estimate \eqref{eq:smooest} for the homogeneous
flow $e^{-itH}$. By similar methods, it is not difficult to see
that $H$-supersmoothness is equivalent to a
\emph{nonhomogeneous} estimate
(compare also with \cite{Mochizuki10-a}):

\begin{theorem}\label{the:2}
  Let $A:\mathcal{H}\to \mathcal{H}_{1}$ be a closed operator
  with dense domain $D(A)$. 
  Assume $A$ is $H$-supersmooth with constant $a$.
  Then
  $e^{-itH}v\in D(A)$ for almost any $t\in \mathbb{R}$
  and any $v\in \mathcal{H}$; moreover,
  for any step function 
  $h(t):\mathbb{R}\to D(A^{*})$,
  $Ae^{-i(t-s)H}A^{*}h(s)$ is Bochner integrable in $s$ 
  over $[0,t]$ (or $[t,0]$)
  and satisfies, for all $\epsilon\in(-\epsilon_{0},\epsilon_{0})$,
  the estimate
  \begin{equation}\label{eq:ssmooest}
    \textstyle
    \|e^{-\epsilon t}\int_{0}^{t}Ae^{-i(t-s)H}A^{*}h(s)ds\|
    _{L^{2}\mathcal{H}_{1}}
    \le 2a
    \|e^{-\epsilon t}h(t)\|_{L^{2}\mathcal{H}_{1}}.
  \end{equation}
  Conversely, if \eqref{eq:ssmooest} holds, then $A$
  is $H$-supersmooth with constant $2a$.
\end{theorem}

\begin{proof}
  Assume $A$ is $H$-supersmooth, hence in particular $H$-smooth.
  Then from the previous Theorem we know that $e^{-itH}v\in D(A)$
  for a.e.~$t$.
  Denote for a function $v(t):\mathbb{R}\to\mathcal{H}$
  its Laplace transforms by
  \begin{equation*}
    \textstyle
    \widetilde{v}_{+}(z)=    
    \int_{0}^{+\infty}e^{izt}v(t)dt,\quad
    \Im z>0
    \qquad\text{and}\qquad 
    \widetilde{v}_{-}(z)=  
    \int_{-\infty}^{0}e^{izt}v(t)dt,\quad
    \Im z<0.
  \end{equation*}
  Note that both integrals converge if
  $\|v(t)\|_{\mathcal{H}}$ grows at most polynomially.
  In particulat, if $F(t)\in L^{\infty}(\mathbb{R};\mathcal{H})$
  and
  \begin{equation}\label{eq:defv}
    \textstyle
    v(t)=\int_{0}^{t}e^{-i(t-s)H}F(s)ds,
  \end{equation}
  we have the well known identities
  \begin{equation}\label{eq:transf}
    \widetilde{v}_{\pm}(z)=
    i^{-1}R(z)\widetilde{F}(z),
    \qquad
    \pm\Im z>0.
  \end{equation}
  Now
  let $h(t):\mathbb{R}\to D(A^{*})$ be a step function, define
  \begin{equation*}
    \textstyle
    v(t)=\int_{0}^{t}e^{-i(t-s)H}A^{*}h(s)ds
  \end{equation*}
  and consider the Laplace transforms of $v(t)$; we have
  by \eqref{eq:transf}
  \begin{equation*}
    \widetilde{v}_{\pm}(z)=i^{-1}R(z)\widetilde{A^{*}h}(z)
    =
    i^{-1}R(z)A^{*}\widetilde{h}(z)
    ,
    \qquad
    \pm\Im z>0
  \end{equation*}
  where $\widetilde{A^{*}h}=A^{*}\widetilde{h}$ follows by
  Hille's theorem 
  (Theorem 3.7.12 in \cite{HillePhillips74-a}). 

  Note also that $v(t)\in D(A)$ for all $t$. To see this,
  write explicitly
  \begin{equation*}
    \textstyle
    h(t)=\sum_{j=1}^{N}\one{E_{j}}(t)h_{j}
  \end{equation*}
  for some $h_{j}\in D(A^{*})$ and some measurable disjoint bounded
  sets $E_{j}\subset \mathbb{R}$ ($\one{E}$ denotes the
  characteristic function of $E$). Then we have, for fixed $t$,
  \begin{equation*}
    \textstyle
    v(t)=\sum\int_{0}^{t}\one{E_{j}}(s)e^{is H}v_{j}ds,
    \qquad
    v_{j}=e^{-itH}h_{j}
  \end{equation*}
  and we know that 
  $Ae^{isH}v_{j}\in L^{2}(0,t;\mathcal{H}_{1})\subset 
  L^{1}(0,t;\mathcal{H}_{1})$ by the previous Theorem.
  Thus by Hille's theorem we deduce that $v(t)\in D(A)$
  and
  \begin{equation}\label{eq:hille}
    \textstyle
    Av(t)=\int_{0}^{t}Ae^{-i(t-s)H}A^{*}h(s)ds.
  \end{equation}

  Now take a second step function
  $g(t):\mathbb{R}\to D(A^{*})$ and apply Parseval's
  identity: for all $\epsilon>0$,
  \begin{equation}\label{eq:pars}
    \textstyle
    2\pi
    \int_{-\infty}^{+\infty}(\widetilde{v}_{\pm}
          (\lambda\pm i \epsilon),
    \widetilde{A^{*}g}(\lambda\pm i \epsilon))_{\mathcal{H}}
       d \lambda
    =
    \pm\int_{0}^{\pm\infty}e^{-2\epsilon|t|}
    (v(t),A^{*}g(t))_{\mathcal{H}}dt.
  \end{equation}
  Using again Hille's theorem to prove 
  $\widetilde{A^{*}g}=A^{*}\widetilde{g}$, and the
  supersmoothness assumption (see also Remark \ref{rem:normal}) 
  we can write
  \begin{equation*}
    |(\widetilde{v}_{\pm}(z),\widetilde{A^{*}g}(z))_{\mathcal{H}}|
    =
    |(AR(z)A^{*}\widetilde{h}(z),\widetilde{g}(z))_{\mathcal{H}_{1}}|
    \le
    2a\|\widetilde{h}(z)\|_{\mathcal{H}_{1}}
    \|\widetilde{g}(z)\|_{\mathcal{H}_{1}}
  \end{equation*}
  and plugging into \eqref{eq:pars} we obtain
  with a last application of Parseval
  \begin{equation*}
  \begin{split}
    \textstyle
    |\int_{0}^{\pm\infty}
    e^{-2\epsilon|t|}
    (v(t),A^{*}g(t))_{\mathcal{H}}dt|
    \le &
    \textstyle
    4\pi
    a
    (\int_{-\infty}^{+\infty}
      \|\widetilde{h}(\lambda\pm i \epsilon)\|_{\mathcal{H}_{1}}^{2}
      d \lambda)^{\frac12}
    (\int_{-\infty}^{+\infty}
      \|\widetilde{g}(\lambda\pm i \epsilon)\|_{\mathcal{H}_{1}}^{2}
      d \lambda)^{\frac12}
    \\
    = &
    2a
    \|e^{-\epsilon|t|}h\|_{L^{2}(\mathbb{R}^{\pm};\mathcal{H}_{1})}
    \|e^{-\epsilon|t|}g\|_{L^{2}(\mathbb{R}^{\pm};\mathcal{H}_{1})}.
  \end{split}
  \end{equation*}
  Recalling that
  $v(t)\in D(A)$ we arrive at
  \begin{equation*}
    \textstyle
    |\int_{0}^{\pm\infty}
    (e^{-\epsilon t}Av(t),e^{-\epsilon t}g(t))_{\mathcal{H}_{1}}dt|
    \le 
    2a
    \|e^{-\epsilon t}h\|_{L^{2}(\mathbb{R}^{\pm};\mathcal{H}_{1})}
    \|e^{-\epsilon t}g\|_{L^{2}(\mathbb{R}^{\pm};\mathcal{H}_{1})}.
  \end{equation*}
  By density of step functions (and of $D(A^{*})$) this implies
  \begin{equation*}
    \textstyle
    \left\|e^{-\epsilon t}Av \right\|_{L^{2}\mathcal{H}_{1}}
    \le 2a\|e^{-\epsilon t}h\|_{L^{2}\mathcal{H}_{1}}
  \end{equation*}
  and recalling \eqref{eq:hille} we obtain
  \eqref{eq:ssmooest} (including the case $\epsilon=0$ which is
  obtained by taking the limit $\epsilon\to0$).

  We now prove the converse statement. Assume
  \eqref{eq:ssmooest} holds for any step function
  $h:\mathbb{R}^{+}\to D(A^{*})$; then by a simple
  approximation argument one sees that \eqref{eq:ssmooest}
  holds for any function $h(t)$ of the form
  \begin{equation*}
    \textstyle
    h(t)=\sigma(t)h_{0},\qquad
    h_{0}\in D(A^{*}),\quad
    \sigma(t)\in L^{2}(\mathbb{R}^{+});
  \end{equation*}
  thus, writing
  \begin{equation*}
    \textstyle
    w(t)=\int_{0}^{t}Ae^{-i(t-s)H}A^{*}h(s)ds
  \end{equation*}
  and applying \eqref{eq:ssmooest} we get
  \begin{equation*}
    \textstyle
    \|e^{-\epsilon t}w(t)\|_{L^{2}(\mathbb{R}^{+};\mathcal{H}_{1})}
    \le 2a \|e^{-\epsilon t}h\|_{L^{2}\mathcal{H}_{1}}.
  \end{equation*}
  By Parseval we have then
  \begin{equation*}
    (2\pi)^{\frac12}
    \|\widetilde{w}_{+}(\lambda+i \epsilon)\|
    _{L^{2}_{\lambda}\mathcal{H}_{1}}=
    \|e^{-\epsilon t}w(t)\|_{L^{2}(\mathbb{R}^{+};\mathcal{H}_{1})}
    \le 2a \|e^{-\epsilon t}h\|_{L^{2}\mathcal{H}_{1}}
    =
    (2\pi)^{\frac12}2a
    \|\widetilde{h}(\lambda+i \epsilon)\|
    _{L^{2}_{\lambda}\mathcal{H}_{1}}
  \end{equation*}
  so that
  \begin{equation}\label{eq:part}
    \|\widetilde{w}_{+}(\lambda+i \epsilon)\|
    _{L^{2}_{\lambda}\mathcal{H}_{1}}\le
    2a
        \|\widetilde{h}(\lambda+i \epsilon)\|
        _{L^{2}_{\lambda}\mathcal{H}_{1}}.
  \end{equation}
  Recalling \eqref{eq:transf}, we have
  \begin{equation*}
    \widetilde{w}_{+}(\lambda+i \epsilon)=
    i^{-1}AR(\lambda+i \epsilon)A^{*}
    \widetilde{h}_{+}(\lambda+i \epsilon)
  \end{equation*}
  where
  \begin{equation*}
    \textstyle
    \widetilde{h}_{+}(\lambda+i \epsilon)=
    \int_{0}^{\infty}e^{i (\lambda+i \epsilon) t}
       \sigma(t)dt \cdot h_{0}
    =
    \widetilde{\sigma}_{+}(\lambda+i \epsilon)\cdot h_{0}
  \end{equation*}
  thus we can write
  \begin{equation*}
    \textstyle
    \|\widetilde{w}_{+}(\lambda+i \epsilon)\|
    _{L^{2}_{\lambda}\mathcal{H}_{1}}^{2}=
    \int_{-\infty}^{\infty}
       \phi(\lambda+i \epsilon)
       \left|
         \widetilde{\sigma}(\lambda+i \epsilon)
       \right|^{2}d \lambda,
      \qquad
      \phi(\lambda)=
      \|AR(\lambda+i \epsilon)A^{*}h_{0}\|_{\mathcal{H}_{1}}^{2}
  \end{equation*}
  Plugging into \eqref{eq:part} we obtain
  \begin{equation*}
    \textstyle
    \int_{-\infty}^{\infty}
       \phi(\lambda+i \epsilon)
       \left|
         \widetilde{\sigma}(\lambda+i \epsilon)
       \right|^{2}d \lambda
    \le
    (2a)^{2}\|h_{0}\|^{2}_{\mathcal{H}_{1}}
    \int_{-\infty}^{\infty}
    |\widetilde{\sigma}(\lambda+i \epsilon)|^{2}d \lambda
  \end{equation*}
  and recalling that $\sigma\in L^{2}$ is arbitrary, we deduce
  \begin{equation*}
    \sup_{\lambda\in \mathbb{R}}|\phi(\lambda+i \epsilon)|
    \le 2a \|h_{0}\|_{\mathcal{H}_{1}}
  \end{equation*}
  which is precisely the $H$-supersmoothness condition for $A$.
\end{proof}

\begin{remark}\label{rem:pettis}
  It would be possible to prove a more general result where
  $h$ is taken to be a generic function in $L^{2}\mathcal{H}_{1}$
  with values in $D(A^{*})$, instead of a step function.
  However this makes the proof of \eqref{eq:hille} 
  rather involved (in particular, the integral in \eqref{eq:hille}
  must be interpreted in the sense of Pettis).
  Since in concrete applications the final approximation
  step becomes trivial,
  we opted for a simpler statement expressed in terms of step
  functions.
\end{remark}

We now show that the smoothness property is inherited
by the square root of $H$:

\begin{theorem}\label{the:wave}
  Let $\nu\in \mathbb{R}$ with $H+\nu\ge0$ and
  $H+\nu$ injective.
  Assume $A$ and $A(H+\nu)^{-\frac14}$ are closed operators
  with dense domain from $\mathcal{H}$ to $\mathcal{H}_{1}$.
  \\
  (i) If $A$ is $H$-smooth with constant $a$, then
  $A(H+\nu)^{-\frac14}$ is $\sqrt{H+\nu}$-smooth
  with constant $C=(\pi+3)a$.
  In particular, we have the estimate
  \begin{equation}\label{eq:sqsmoo}
    \|Ae^{-it \sqrt{H+\nu}}v\|_{L^{2}\mathcal{H}_{1}}
    \le
    2C^{\frac12}\|(H+\nu)^{\frac14}v\|_{\mathcal{H}},
    \qquad
    \forall v\in D((H+\nu)^{\frac14}).
  \end{equation}
  (ii) If $A$ is $H$-supersmooth with constant $a$, then
  $A(H+\nu)^{-\frac14}$ is $\sqrt{H+\nu}$-supersmooth
  with constant $C=(\pi+3)a$.
  In particular, we have the estimate
  \begin{equation}\label{eq:sqssmoo}
    \textstyle
    \|\int_{0}^{t}Ae^{-i(t-s)\sqrt{H+\nu}}
      (H+\nu)^{-\frac12}A^{*}h(s)ds\|
      _{L^{2}\mathcal{H}_{1}}
    \le
    2(\pi+3)a\|h\| _{L^{2}\mathcal{H}_{1}}
  \end{equation}
  for any step function 
  $h:\mathbb{R}\to D((H+\nu)^{-\frac14}A^{*})$.
\end{theorem}

\begin{proof}
  We give a detailed proof for case (ii) and at the end
  we shall list the (minor) modifications needed
  to prove (i). Note that by renaming the operator $H$, it is not
  restrictive to assume $\nu=0$.

  We need to prove a uniform bound in $\epsilon_{0}>\Im z>0$ 
  for the operators
  \begin{equation}\label{eq:0}
    AH^{-\frac14}(\sqrt{H}-z)^{-1}H^{-\frac14}A^{*}
    :\mathcal{H}_{1}\to \mathcal{H}_{1}.
  \end{equation}
  We have (the notation $S \subset T$ 
  means that the operator $T$ extends $S$)
  \begin{equation*}
    H^{-\frac14}(\sqrt{H}-z)^{-1}H^{-\frac14}
    \subset
    H^{-\frac12}(\sqrt{H}+z)(H-z^{2})^{-1}
    =
    (I+zH^{-\frac12})R(z^{2}).
  \end{equation*}
  By assumption we already know that $AR(z^{2})A^{*}$ is
  uniformly bounded with norm $\le a$, thus it remains to
  prove that
  \begin{equation}\label{eq:1}
    AzH^{-\frac12}R(z^{2})A^{*}:\mathcal{H}_{1}\to \mathcal{H}_{1}
  \end{equation}
  is uniformly bounded. We plan to estimate this
  operator using the spectral theorem. The obstruction to
  such an estimate is due to a singularity at $\lambda=z^{2}$
  in the spectral representation (see below); however, we
  can remove this singularity by adding (or subtracting)
  a suitable operator for which we have already an estimate.
  Indeed, using the resolvent identity
  \begin{equation*}
    R(z^{2})-R(-|z|^{2})=(z^{2}+|z|^{2})R(z^{2})R(|z|^{2})
  \end{equation*}
  we see, again by assumption, that the operator
  \begin{equation*}
    A(z^{2}+|z|^{2})R(z^{2})R(|z|^{2})A^{*}
  \end{equation*}
  is also bounded with norm $\le 2a$. 
  Adding or
  subtracting this from \eqref{eq:1}, we see that it
  is sufficient to prove that at least one of the operators
  \begin{equation*}
    Q_{\pm}(z)=
    A\left[
    zH^{-\frac12}\pm
    (z^{2}+|z|^{2})R(|z|^{2})
    \right]R(z^{2})A^{*}
  \end{equation*}
  is bounded uniformly in $z$. By Stone's
  formula we can write
  \begin{equation*}
    \textstyle
    Q_{\pm}(z)v=\lim_{\epsilon \downarrow0}
      \int_{0}^{\infty}
      \frac{1}{\lambda-z^{2}}
      \left(
         \frac{z}{\sqrt{\lambda}}\pm
         \frac{z^{2}+|z|^{2}}{\lambda+|z|^{2}}
      \right)
      A\Im R(\lambda+i \epsilon)A^{*}v d \lambda
  \end{equation*}
  and using one last time the assumption on the uniform bound
  $\le a$ for the norm of the operator
  $A\Im R(\lambda+i \epsilon)A^{*}$, we obtain
  \begin{equation*}
    \textstyle
    \|Q_{\pm}(z)v\|_{\mathcal{H}_{1}}
    \le
    \int_{0}^{\infty}
    \frac{1}{|\lambda-z^{2}|}
    \left|
      \frac{z}{\sqrt{\lambda}}\pm
      \frac{z^{2}+|z|^{2}}{\lambda+|z|^{2}}
    \right|
    d \lambda
    \cdot
    a\|v\|_{\mathcal{H}_{1}}.
  \end{equation*}
  We now check that the last integral is uniformly bounded
  for $\Im z>0$. Let $z=e^{i \theta}|z|$, with $|z|>0$
  and $0<\theta<\pi$, then by rescaling $\lambda=|z|^{2}\mu$ we have
  \begin{equation*}
    \textstyle
    I_{\pm} \equiv
    \int_{0}^{\infty}
    \frac{1}{|\lambda-z^{2}|}
    \left|
      \frac{z}{\sqrt{\lambda}}\pm
      \frac{z^{2}+|z|^{2}}{\lambda+|z|^{2}}
    \right|
    d \lambda
    =
    \int_{0}^{\infty}
    \frac{1}{|\mu-e^{2i \theta}|}
    \left|
      \frac{e^{i \theta}}{\sqrt{\mu}}\pm
      \frac{e^{2i \theta}+1}{\mu+1}
    \right|
    d \mu.
  \end{equation*}
  In the case $0<\theta\le \frac\pi2$, we
  prove that $I_{-}$ is bounded. Indeed,
  using the identity
  \begin{equation*}
    |e^{i \theta}(\mu+1)-\sqrt{\mu}(e^{2i \theta}+1)|=
    |\sqrt{\mu}e^{i \theta}-1|\cdot|\sqrt{\mu}-e^{i \theta}|=
    |\sqrt{\mu}-e^{i \theta}|^{2}
  \end{equation*}
  we obtain
  \begin{equation*}
    \textstyle
    I_{-}=
    \int_{0}^{\infty}
    \frac{1}{|\mu-e^{2i \theta}|}
    \frac{|\sqrt{\mu}-e^{i \theta}|^{2}}{\sqrt{\mu}(\mu+1)}
    d \mu.
  \end{equation*}
  We then simplify the first factor in $|\mu-e^{2i \theta}|=
  |\sqrt{\mu}-e^{i \theta}|\cdot|\sqrt{\mu}+e^{i \theta}|$
  to obtain
  \begin{equation*}
    \textstyle
    I_{-}=
    \int_{0}^{\infty}
    \frac{|\sqrt{\mu}-e^{i \theta}|}
       {\sqrt{\mu}(\mu+1)\cdot|\sqrt{\mu}+e^{i \theta}|}
    d \mu
    \le
    \int_{0}^{\infty}
    \frac{1}
       {\sqrt{\mu}(\mu+1)}
    d \mu=\pi.
  \end{equation*}
  Thus if $0<\theta\le\frac\pi2$ we have
  $\|Q_{-}(z)v\|_{\mathcal{H}_{1}}\le \pi a \|v\|_{\mathcal{H}_{1}}$;
  this implies a bound $(\pi+2)a$ for the norm of the operator
  \eqref{eq:1} for such $z$, and a bound $(\pi+3)a$ for the norm
  of \eqref{eq:0}.

  On the other hand, in the case $\frac{\pi}{2}<\theta<\pi$
  we prove that $I_{+}$ is bounded: we have
  \begin{equation*}
    \textstyle
    I_{+}=
    \int_{0}^{\infty}
    \frac{|(\mu+1)e^{i \theta}+\sqrt{\mu}(e^{2i \theta}+1)|}
    {\sqrt{\mu}(\mu+1)|\mu-e^{2i \theta}|}
    d\mu
  \end{equation*}
  and using the identity
  \begin{equation*}
    |(\mu+1)e^{i \theta}+\sqrt{\mu}(e^{2i \theta}+1)|=
    |\sqrt{\mu}+e^{i \theta}|\cdot|\sqrt{\mu}e^{i \theta}+1|=
    |\sqrt{\mu}+e^{i \theta}|^{2}
  \end{equation*}
  and simplifying the first factor in
  \begin{equation*}
    |\mu-e^{2i \theta}|=
    |\sqrt{\mu}+e^{i \theta}|\cdot|\sqrt{\mu}-e^{i \theta}|
  \end{equation*}
  we obtain
  \begin{equation*}
    \textstyle
    I_{+}=
    \int_{0}^{\infty}
    \frac{|\sqrt{\mu}+e^{i \theta}|}
    {\sqrt{\mu}(\mu+1)|\sqrt{\mu}-e^{i \theta}|}
    d\mu
    \le
    \int_{0}^{\infty}
    \frac{1}
       {\sqrt{\mu}(\mu+1)}
    d \mu=\pi.
  \end{equation*}
  Thus we obtain the same bound as before for the
  remaining values of $z$, proving that
  \eqref{eq:0} is $\sqrt{H}$-supersmooth with a constant
  $(\pi+3)a$. The final estimate
  \eqref{eq:sqssmoo} is a direct consequence of Theorem
  \ref{the:2}.

  It is easy to modify the previous argument for the proof of (i):
  indeed, the bound for
  \begin{equation*}
    AH^{-\frac14}\Im (\sqrt{H}-z)^{-1}H^{-\frac14}A^{*}
  \end{equation*}
  is reduced as before to bounds for the operators
  \begin{equation*}
    \widetilde{Q}_{\pm}(z)=
    A\Im\left\{\left[
    zH^{-\frac12}\pm
    (z^{2}+|z|^{2})R(|z|^{2})
    \right]R(z^{2})\right\}
    A^{*}
  \end{equation*}
  which follow exactly from the computations above.
  Finally applying Theorem \ref{the:1} we obtain
  the estimate
  \begin{equation*}
    \|Ae^{-it \sqrt{H+\nu}}(H+\nu)^{-\frac14}v\|
        _{L^{2}\mathcal{H}_{1}}
    \le
    2C^{\frac12}\|v\|_{\mathcal{H}},
  \end{equation*}
  whence we obtain \eqref{eq:sqsmoo}.
\end{proof}

It is not difficult to see that the injectivity assumption
on $H+\nu$ is not necessary, in the following sense. 
If $\ker(H+\nu)\neq\{0\}$,
denoting by $P$ be the orthogonal projection
onto $\mathcal{K}=\ker(H+\nu)^{\perp}$ and by 
$(H+\nu)^{-\frac14}$ the operator 
$(H+\nu)\vert_{\mathcal{K}}^{-\frac14}$, we see that
the operator $(H+\nu)^{-\frac14}P$
is closed and densely defined on $\mathcal{H}$.
Moreover,
if $A$ is $H$-smooth and $v\in\ker(H+\nu)$, 
from the smoothing estimate \eqref{eq:smooest} it follows 
immediately that $Av=0$. Thus we
obtain the following more general result:

\begin{corollary}\label{cor:waveP}
  Let $\nu\in \mathbb{R}$ with $H+\nu\ge0$ and let $P$
  be the orthogonal projection onto $\ker(H+\nu)^{\perp}$.
  Assume $A$ and $A(H+\nu)^{-\frac14}P$ are closed operators with
  dense domain from $\mathcal{H}$ to $\mathcal{H}_{1}$.
  \\
  (i) If $A$ is $H$-smooth with constant $a$, then
  $A(H+\nu)^{-\frac14}P$ is $\sqrt{H+\nu}$-smooth
  with constant $C=(\pi+3)a$.
  In particular, we have the estimate
  \begin{equation}\label{eq:sqsmoo2}
    \|Ae^{-it \sqrt{H+\nu}}v\|_{L^{2}\mathcal{H}_{1}}
    \le
    2C^{\frac12}\|(H+\nu)^{\frac14}v\|_{\mathcal{H}},
    \qquad
    \forall v\in D((H+\nu)^{\frac14}).
  \end{equation}
  (ii) If $A$ is $H$-supersmooth with constant $a$, then
  $A(H+\nu)^{-\frac14}P$ is $\sqrt{H+\nu}$-supersmooth
  with constant $C=(\pi+3)a$.
  In particular, we have the estimate
  \begin{equation}\label{eq:sqssmoo2}
    \textstyle
    \|\int_{0}^{t}Ae^{-i(t-s)\sqrt{H+\nu}}
      (H+\nu)^{-\frac12}PA^{*}h(s)ds\|
      _{L^{2}\mathcal{H}_{1}}
    \le
    C\|h\| _{L^{2}\mathcal{H}_{1}}
  \end{equation}
  for any step function 
  $h:\mathbb{R}\to D((H+\nu)^{-\frac14}PA^{*})$.
\end{corollary}

\begin{proof}
  The proof is obtained simply by restricting to the closed
  subspace $\ker(H+\nu)^{\perp}$ and applying the previous
  Theorem. Note that in case (i) we get the estimate
  \begin{equation*}
    \|Ae^{-it \sqrt{H+\nu}}(H+\nu)^{-\frac14}v\|
        _{L^{2}\mathcal{H}_{1}}
    \le
    2C^{\frac12}\|v\|_{\mathcal{H}},
    \qquad
    \forall v\in \ker(H+\nu)^{\perp}
  \end{equation*}
  which gives \eqref{eq:sqsmoo} for $v\in \ker(H+\nu)^{\perp}$,
  while for $v\in \ker(H+\nu)$ estimate \eqref{eq:sqsmoo}
  is trivial since the left hand side is identically 0
  by the remark preceding the Theorem.
\end{proof}

\begin{remark}\label{rem:fractional}
  By the same technique it is possible to prove
  smoothing estimates for the flows of the form
  $e^{it(H+\nu)^{1/m}}$, $m\ge0$ integer; this might
  be interesting especially for $m=4$ since the equation
  $iu_{t}+|D|^{\frac12}u=0$ is relevant in connection with
  water waves. We omit the details.
\end{remark}

\section{Applications}\label{sec:applications}

The results of the previous Section are rather general and
have a wide range of applications;
here we shall mention just a few. We use the notations
\begin{equation*}
  |D|^{\alpha}=(-\Delta)^{\frac a2},\qquad
  \bra{D}^{\alpha}=(1-\Delta)^{\frac\alpha2},\qquad
  \bra{x}=(1+|x|^{2})^{\frac12}.
\end{equation*}
The operators appearing in the following are
intended to be closed extension of the corresponding
operators defined on $C^{\infty}_{c}(\mathbb{R}^{n})$.

\subsection{Powers of the Laplacian}
\label{sub:powers_of_the_laplacian}


We begin by recalling a few smoothing
estimates for operators with constant coefficients.
Most of these results
are well known or can be obtained directly
via Fourier analysis; we review them
both for the purpose of illustration and for later use below.

In the Kato-Yajima paper \cite{KatoYajima89-a}, the case
$H=-\Delta$ on $\mathbb{R}^{n}$, $n\ge3$
is considered in detail, and in particular it is proved that
\begin{itemize}
    \item $|x|^{\alpha-1}|D|^{\alpha}$ for $0\le \alpha<1/2$
    is $(-\Delta)$-supersmooth 
    \item $\bra{x}^{-1}\bra{D}^{\frac12 }$
    is $(-\Delta)$-supersmooth 
    \item the multiplication operator by any
    $a(x)\in L^{n}(\mathbb{R}^{n})$
    is $(-\Delta)$-supersmooth 
\end{itemize}
This result is precised in \cite{Watanabe91-a} as follows:
for $n\ge2$,
\begin{itemize}
    \item $|x|^{-\beta}|D|^{\alpha-\beta}$ is
    $|D|^{2\alpha}$-supersmooth for
    $2\alpha>1$, $\beta\le \alpha$ and $\frac12<\beta<\frac n2$.
\end{itemize}
If insted of the supersmoothing property we restrict
to the weaker smoothing property,
several additional results are available
(\cite{Walther99-a}, \cite{RuzhanskySugimoto04-a} among
the others).
We mention in particular the following one, which was
proved in \cite{Ben-ArtziKlainerman92-a}, \cite{Chihara02-a}
and will be used below:
\begin{itemize}
    \item $\bra{x}^{-\frac12-\epsilon}|D|^{\frac12}$ is
    $(-\Delta)$-smooth for $n\ge2$, $\epsilon>0$.
\end{itemize}

The previous results, combined with Theorems 
\ref{the:1}, \ref{the:2}, \ref{the:wave},
give immediately several estimates
for time dependent flows.
By Watanabe's result and Theorem \ref{the:1} we get
\begin{equation*}
  \textstyle
  \||x|^{-\beta}|D|^{\alpha-\beta}e^{-it|D|^{2 \alpha}}f\|
     _{L^{2}(\mathbb{R}^{n+1})}
  \lesssim\|f\|_{L^{2}},\qquad
  2\alpha>1,\quad \beta\le \alpha,\quad \frac12<\beta<\frac n2.
\end{equation*}
This estimate is implied by the result in
\cite{FangWang11-a} (see also \cite{Hoshiro97-a})
\begin{equation*}
  \||x|^{-\beta}|D|^{\alpha-\beta}e^{-it|D|^{2 \alpha}}f\|
     _{L^{2}(\mathbb{R}^{n+1})}
  \lesssim\|\Lambda^{\frac12-\beta} f\|_{L^{2}},
  \qquad
  \Lambda=(1-\Delta_{\mathbb{S}^{n-1}})^{\frac12}
\end{equation*}
valid for any $\alpha>0$ and $\beta\in(\frac12,\frac n2)$, 
with a further
gain in angular regularity (note that $\frac12-\beta<0$).
However, our theory covers also the nonhomogeneous case.
Indeed, applying Theorem \ref{the:2} we get,
for the same range of parameters as in Watanabe's result,
the nonhomogeneous estimate
\begin{equation}\label{eq:result1}
  \textstyle
  \||x|^{-\beta}|D|^{\alpha-\beta}
     \int_{0}^{t}e^{-i(t-s)|D|^{2 \alpha}}F(s)ds\|
     _{L^{2}(\mathbb{R}^{n+1})}
  \lesssim
  \||x|^{\beta}|D|^{-\alpha+\beta}F\|
     _{L^{2}(\mathbb{R}^{n+1})}.
\end{equation}
Estimate \eqref{eq:result1} does not include the case of
the wave flow since $\alpha>1/2$ in \cite{Watanabe91-a};
however using Theorem \ref{the:wave}
we obtain that \eqref{eq:result1} holds for all $\alpha>1/4$.

Several other applications to constant 
coefficient equations are possible.
As a final example we consider Klein-Gordon equations,
which are covered by Theorem \ref{the:2} with $\mu=1$:
we obtain the smoothing estimate
\begin{equation}\label{eq:KG1}
  \||x|^{-\beta}|D|^{1-\beta}e^{it \bra{D}}f\|
    _{L^{2}(\mathbb{R}^{n+1})}
  \lesssim
  \|\bra{D}^{\frac12}f\|_{L^{2}}.
\end{equation}

\subsection{Potentials of critical decay}
\label{sub:potentials_of_critical_decay}

Our second example is a simplification of a proof in
\cite{BurqPlanchonStalker04-a}, where Strichartz
estimates were obtained for Schr\"{o}dinger and wave equations
of the form
\begin{equation*}
  iu_{t}-\Delta u+V(x)u=0,\qquad
  u_{tt}-\Delta u+V(x)u=0.
\end{equation*}
$V(x)$ is a real valued potential satisfying 
the following assumptions: there exist
$C>0$ and $c<\frac{(n-2)^{2}}{4}$, $n\ge3$, such that
\begin{equation}\label{eq:assVb}
  \frac{C}{|x|^{2}}\ge V(x)\ge -\frac{c}{|x|^{2}}
  \quad\text{and}\quad
   -\partial_{r}(|x|V(x))\ge -\frac{c}{|x|^{2}},
   \qquad
   \partial_{r}:=\frac{x}{|x|}\cdot \nabla_{x}
\end{equation}
(the actual assumptions are slightly more general).
The crucial step in \cite{BurqPlanchonStalker04-a} is Theorem 3,
claiming that
\begin{equation}\label{eq:supb}
  |x|^{-1}\ \text{is $(-\Delta+V)$-supersmoothing}.
\end{equation}
The standard Kato theory (Theorem \ref{the:1} here) 
gives a smoothing estimate for the Schr\"{o}dinger flow:
\begin{equation}\label{eq:smoob}
  \||x|^{-1}e^{it(-\Delta-V)}f\|
     _{L^{2}(\mathbb{R}^{n+1})}
  \lesssim\|f\|_{L^{2}}.
\end{equation}
Then the full set of Strichartz estimates follows
from \eqref{eq:smoob}, via the usual Rodnianski-Schlag
trick \cite{RodnianskiSchlag04-a}.

In order to apply the same procedure to the wave equation,
in \cite{BurqPlanchonStalker04-a} the following estimate
for the wave flow is proved:
\begin{equation}\label{eq:waveburq}
  \||x|^{-1}e^{it \sqrt{-\Delta-V}}f\|
     _{L^{2}(\mathbb{R}^{n+1})}
  \lesssim\|f\|_{\dot H^{\frac12}}.
\end{equation}
The proof in \cite{BurqPlanchonStalker04-a}
is rather involved (see also
the Errata relative to that paper); 
however, using the theory developed in Section \ref{sec:kato},
\eqref{eq:waveburq} follows directly by
the combination of \eqref{eq:supb} and \eqref{eq:sqsmoo}.

Note that we can prove additional estimates
which are apparently new. For instance, we
deduce the following homogeneous estimate
for the Klein-Gordon flow
\begin{equation}\label{eq:KG2}
  \||x|^{-1}e^{it \sqrt{1-\Delta-V}}f\|
     _{L^{2}(\mathbb{R}^{n+1})}
  \lesssim\|f\|_{H^{\frac12}},
\end{equation}
and nonhomogenous estimates
for all the flows, like
\begin{equation*}
  \textstyle
  \||x|^{-1}\int_{0}^{t}e^{-i(t-s)(-\Delta+V)}F(s)ds\|
     _{L^{2}(\mathbb{R}^{n+1})}
  \lesssim\||x|F\|
     _{L^{2}(\mathbb{R}^{n+1})}
\end{equation*}
for the Schr\"{o}dinger equation
and similar ones for the wave and Klein-Gordon equations.

\subsection{Wave equation with large magnetic potentials}
\label{sub:wave_equation_with_large}

We conclude the paper 
by proving Strichartz estimates for a wave
equation on $\mathbb{R}^{n}$, $n\ge3$, of the form
\begin{equation*}
  u_{tt}+Hu=0
\end{equation*}
where $H$ is a magnetic Schr\"{o}dinger operator
\begin{equation}\label{eq:H}
  H=(i \nabla+A(x))^{2}+V(x).
\end{equation}
We make the following assumptions: $A(x)\in \mathbb{R}^{n}$,
$V(x)\in \mathbb{R}$, moreover for some $C,\epsilon_{0}>0$
\begin{equation}\label{eq:ass1AV}
  |A(x)|+\bra{x}|V(x)|\le C \bra{x}^{-1-\epsilon_{0}},
\end{equation}
\begin{equation}\label{eq:ass2AV}
  \forall\  0<\epsilon_{1}<\epsilon_{0}, \quad
  \bra{x}^{1+\epsilon}A(x)\in \dot H^{\frac12}_{2n}(\mathbb{R}^{n})
  \cap
  C^{0}(\mathbb{R}^{n})
\end{equation}
where $\dot H^{\frac12}_{q}$ is the space with norm 
$\||D|^{\frac12}u\|_{L^{q}}$, $|D|=(-\Delta)^{\frac12}$;
finally, we assume that 0 is not an eigenvalue of $H$
nor a resonance, in the sense that
\begin{equation}\label{eq:reson}
  \bra{x}^{\frac{n-4}{2}-}u(x)\in L^{2}(\mathbb{R}^{n}),
  \qquad
  Hu=0
  \quad\implies\quad
  u \equiv0.
\end{equation}
Clearly in dimension $n\ge5$ it is sufficient to assume that
0 is not an eigenvalue.
Then the results in \cite{ErdoganGoldbergSchlag09-a}
imply the following resolvent estimate:

\begin{theorem}\label{the:EGS}
  Assume the operator $H$ in \eqref{eq:H}
  is selfadjoint, nonnegative, and
  satisfies \eqref{eq:ass1AV}, \eqref{eq:ass2AV}
  and \eqref{eq:reson}. Then the
  resolvent $R(z)=(H-z)^{-1}$ satisfies
  for $\delta>0$
  \begin{equation}\label{eq:reserd1}
    \|\bra{x}^{-\frac12-\delta}|D|^{\frac12}
      R(z)
      |D|^{\frac12}\bra{x}^{-\frac12-\delta}f
    \|_{L^{2}}
    +
    \|\bra{x}^{-1-\delta}
      R(z)
      \bra{x}^{-1-\delta}f
    \|_{L^{2}}
    \le C_{\delta}
    \|f\|_{L^{2}}
  \end{equation}
  with a constant uniform in 
  $z\in \mathbb{C}\setminus \mathbb{R}$.
\end{theorem}

\begin{proof}
  The estimate follows from Theorem 1.2 in
  \cite{ErdoganGoldbergSchlag09-a}
  with the choice $\alpha=1/2$. The result is stated there
  in the form of a limit absorption principle i.e.,
  for $z=\lambda^{2}+i0$; the form given here follows
  from the remark that the spectrum is positive
  by assumption, and then
  by a simple application of the Phragm\'{e}n-Lindel\"{o}f
  principle on the upper resp.~lower complex plane.
\end{proof}

In the terminology of Kato-Yajima, 
Theorem \ref{the:EGS} states that
  $\bra{x}^{-\frac12-\delta}|D|^{\frac12}$
  and $\bra{x}^{-1-\delta}$ are
  $H$-supersmoothing operators.
Applying Theorem \ref{the:wave} we get the estimate
\begin{equation}\label{eq:waveH}
  \|\bra{x}^{-\frac12-\delta}|D|^{\frac12}
    e^{it \sqrt{H}}f\|_{L^{2}(\mathbb{R}^{n+1})}
  +
  \|\bra{x}^{-1-\delta}
    e^{it \sqrt{H}}f\|_{L^{2}(\mathbb{R}^{n+1})}
  \lesssim
  \|H^{\frac14}f\|_{L^{2}(\mathbb{R}^{n})}
\end{equation}
for the corresponding wave flow. Then we can prove
the full set of non-endpoint Strichartz estimates:

\begin{theorem}\label{the:wavestr}
  Let $n\ge3$.
  Assume the operator $H$ in \eqref{eq:H}
  is selfadjoint, nonnegative, and
  satisfies \eqref{eq:ass1AV}, \eqref{eq:ass2AV}
  and \eqref{eq:reson}. Then the wave flow satisfies the
  non-endpoint Strichartz estimates
  \begin{equation}\label{eq:wavestr}
    \textstyle
    \||D|^{\frac1q-\frac1p}e^{it \sqrt{H}}f\|_{L^{p}L^{q}}
    \lesssim
    \|f\|_{\dot H^{\frac12}},
  \end{equation}
  and
  \begin{equation}\label{eq:wavestr2}
    \textstyle
    \||D|^{\frac1q-\frac1p}
       \sin(t \sqrt{H}) H^{-\frac12}g
       \|_{L^{p}L^{q}}
    \lesssim
    \|g\|_{\dot H^{-\frac12}},
  \end{equation}
  for all $2<p\le \infty$, $2\le q<\frac{2(n-1)}{(n-3)}$
  with $2p^{-1}+(n-1)q^{-1}=(n-1)2^{-1}$.
\end{theorem}

\begin{remark}\label{rem:endp}
  It is possible to prove the endpoint estimate
  using the following nonhomogeneous mixed 
  Strichartz-smoothing estimate for the free wave flow
  ($n\ge4$)
  \begin{equation}\label{eq:IK}
    \textstyle
    \|\int_{0}^{t}e^{i(t-s)|D|}F(s)\|_{L^{2}L^{\frac{2(n-1)}{n-3}}}
    \lesssim
    \|\bra{x}^{\frac12+}|D|^{\frac12}F\|_{L^{2}L^{2}}.
  \end{equation}
  This estimate can be obtained by a modification of the
  techniques used in \cite{IonescuKenig05-a} for the
  corresponding result for the Schr\"{o}dinger equation.
  We prefer to omit the details here.
\end{remark}

\begin{proof}[Proof of Theorem \ref{the:wavestr}]
  The function $u=e^{it \sqrt{H}}f$ solves the Cauchy problem
  \begin{equation*}
    u_{tt}-\Delta u=
    -i \nabla \cdot(Au)-iA \cdot \nabla u-(|A|^{2}+V)u,
    \qquad
    u(0,x)=f,\quad
    u_{t}(0,x)=i \sqrt{H}f
  \end{equation*}
  hence we can write
  \begin{equation}\label{eq:tofree}
    e^{it \sqrt{H}}f=\cos(t|D|)f+i\sin(t|D|)|D|^{-1}\sqrt{H}f
               -i\widetilde{I}-i\widetilde{II}-\widetilde{III}
  \end{equation}
  where
  \begin{equation*}
    \textstyle
    \widetilde I=
    \int_{0}^{t}\sin((t-s)|D|)|D|^{-1}\nabla \cdot(Au)ds,
    \qquad
    \widetilde{II}=
    \int_{0}^{t}\sin((t-s)|D|)|D|^{-1}A \cdot\nabla uds
  \end{equation*}
  and
  \begin{equation*}
    \textstyle
    \widetilde{III}=
    \int_{0}^{t}\sin((t-s)|D|)|D|^{-1}(|A|^{2}+V) uds.
  \end{equation*}
  In the following we shall estimate as usual the more general
  operators
  \begin{equation*}
    \textstyle
    I=\int_{0}^{t}e^{i(t-s)|D|}|D|^{-1}\nabla \cdot(Au)ds,
    \qquad
    II=\int_{0}^{t}e^{i(t-s)|D|}|D|^{-1}A \cdot\nabla uds
  \end{equation*}
  and
  \begin{equation*}
    \textstyle
    III=\int_{0}^{t}e^{i(t-s)|D|}|D|^{-1}(|A|^{2}+V) uds
  \end{equation*}
  since the estimates for $I,II,III$ imply
  the corresponding estimates for
  $\widetilde{I},\widetilde{II},\widetilde{III}$.

  Besides the fundamental smoothing estimates
  \eqref{eq:waveH},
  we need the following tools:
  the homogeneous Strichartz estimates for the free
  wave equation
  \begin{equation}\label{eq:freestrich}
    \||D|^{\frac1q-\frac1p}e^{it |D|}f\|_{L^{p}L^{q}}
    \lesssim
    \||D|^{\frac12}f\|_{L^{2}},
  \end{equation}
  the dual smoothing estimate for the free wave equation
  \begin{equation}\label{eq:dualsm}
    \textstyle
    \|\int e^{-is|D|}F(s)ds\|_{L^{2}}
    \lesssim
    \|\bra{x}^{\frac12+}F\|_{L^{2}L^{2}},
  \end{equation}
  and the fact that the operators
  \begin{equation}\label{eq:L2bdd}
    \bra{x}^{\frac12+\epsilon}|D|^{-\frac12}
    \bra{x}^{-\frac12-\epsilon'}|D|^{\frac12}
    \quad\text{and}\quad 
    \bra{x}^{\frac12+\epsilon}|D|^{\frac12}
    \bra{x}^{-\frac12-\epsilon'}|D|^{-\frac12}
  \end{equation}
  are bounded on $L^{2}$ provided 
  $0<\epsilon<\epsilon'$ are small enough. 
  Estimates \eqref{eq:freestrich} are well known;
  \eqref{eq:dualsm} is the dual of
  \begin{equation*}
    \|\bra{x}^{-\frac12-\epsilon}e^{it|D|}f\|_{L^{2}L^{2}}
    \lesssim\|f\|_{L^{2}}
  \end{equation*}
  which follows from Theorem \ref{the:wave} and the
  $-\Delta$-smoothness of 
  $\bra{x}^{-\frac12-\epsilon}|D|^{\frac12}$ proved
  in Section \ref{sub:powers_of_the_laplacian}.
  Finally, the $L^{2}$ boundedness of \eqref{eq:L2bdd}
  can be proved by interpolation, or a direct proof
  can be found in Lemma 6.2 in \cite{ErdoganGoldbergSchlag09-a}.

  Note also that
  by Hardy's inequality, since $|A|\lesssim|x|^{-1}$
  and $|V|\lesssim|x|^{-2}$, we have
  \begin{equation*}
    \|H^{\frac12}f\|_{L^{2}}^{2}
    \le
    \|\nabla u+i Au\|_{L^{2}}^{2}
    +
    \||V|^{\frac12}u\|^{2}_{L^{2}}
    \lesssim\|f\|_{\dot H^{1}}^{2}.
  \end{equation*}
  Thus by interpolation we obtain
  \begin{equation}\label{eq:H14}
    \|H^{\frac14}f\|_{L^{2}}\lesssim\|f\|_{\dot H^{\frac12}}.
  \end{equation}
  Equivalently, the operator 
  $H^{\frac14}|D|^{-\frac12}$ is bounded on $L^{2}$ and hence
  also the operator
  \begin{equation}\label{eq:HD}
    |D|^{-\frac12}H^{\frac12}|D|^{-\frac12}=
    (H^{\frac14}|D|^{-\frac12})^{*}H^{\frac14}|D|^{-\frac12}
  \end{equation}
  is $L^{2}$ bounded.

  The estimate of the homogenous terms in \eqref{eq:tofree}
  follows directly from \eqref{eq:freestrich}; in particular,
  the second term can be estimated by
  \begin{equation*}
    \lesssim
    \||D|^{\frac12}|D|^{-1}\sqrt{H}f\|_{L^{2}}
    =
    \||D|^{-\frac12}H^{\frac12}|D|^{-\frac12}
        |D|^{\frac12}f\|_{L^{2}}
    \lesssim
    \||D|^{\frac12}f\|_{L^{2}}
  \end{equation*}
  using the boundedness of \eqref{eq:HD}.

  We now focus on the main terms $I$, $II$, $III$. By the usual
  application of Christ-Kiselev's Lemma, in the non-endpoint case
  it is sufficient to
  estimate the untruncated integrals which can be split as
  \begin{equation*}
    \textstyle
    e^{it|D|}\int e^{-is|D|}F(s)ds.
  \end{equation*}
  Using first the homogeneous Strichartz estimate
  \eqref{eq:freestrich} to bound $e^{it|D|}$,
  then the dual smoothing estimate \eqref{eq:dualsm},
  we obtain
  \begin{equation*}
    \textstyle
    \||D|^{\frac1q-\frac1p}\int e^{i(t-s)|D|}F(s)ds\|_{L^{p}L^{q}}
    \lesssim
    \|\bra{x}^{\frac12+\epsilon}|D|^{\frac12}F\|_{L^{2}L^{2}}.
  \end{equation*}
  For the term $I$ this gives
  \begin{equation*}
    \||D|^{\frac1q-\frac1p}I\|_{L^{p}L^{q}}
    \lesssim
    \|\bra{x}^{\frac12+\epsilon}|D|^{-\frac12}
      \nabla \cdot(Au)\|_{L^{2}L^{2}}.
  \end{equation*}
  By the boundedness of the Riesz operator $\nabla|D|^{-1}$
  in weighted $L^{2}$ spaces (note that $\bra{x}^{\frac12+}$
  is an $A_{2}$ weight) we get
  \begin{equation*}
    \lesssim
    \|\bra{x}^{\frac12+\epsilon}|D|^{\frac12}
      (Au)\|_{L^{2}L^{2}}.
  \end{equation*}
  Writing for $\epsilon'>\epsilon$
  \begin{equation*}
    \bra{x}^{\frac12+\epsilon}|D|^{\frac12}
    =
    \bra{x}^{\frac12+\epsilon}|D|^{\frac12}
    \bra{x}^{-\frac12-\epsilon'}|D|^{-\frac12}
    \cdot
    |D|^{\frac12}\bra{x}^{\frac12+\epsilon'}
  \end{equation*}
  and recalling \eqref{eq:L2bdd}, we arrive at
  \begin{equation*}
    \lesssim
    \||D|^{\frac12}
      (\bra{x}^{\frac12+\epsilon'} Au)\|_{L^{2}L^{2}}
    =
    \||D|^{\frac12}
      (\bra{x}^{1+2\epsilon'} A \cdot
       \bra{x}^{-\frac12-\epsilon'}u)\|_{L^{2}L^{2}}.
  \end{equation*}
  The Kato-Ponce inequality gives
  \begin{equation*}
  \begin{split}
    \||D|^{\frac12}
      (\bra{x}^{1+2\epsilon'} A \cdot
       \bra{x}^{-\frac12-\epsilon'}u)\|_{L^{2}}
    &
    \lesssim
    \\
    \lesssim
    \|\bra{x}^{1+2\epsilon'} A\|_{\dot H^{\frac12}_{2n}}
    &
    \|\bra{x}^{-\frac12-\epsilon'}u\|_{L^{\frac{2n}{n-1}}}
    +
    \|\bra{x}^{1+2\epsilon'} A\|_{L^{\infty}}
    \||D|^{\frac12}(\bra{x}^{-\frac12-\epsilon'}u)\|_{L^{2}}
  \end{split}
  \end{equation*}
  and by the Sobolev embedding
  $\dot H^{\frac12}\subset L^{\frac{2n}{n-1}}$ we obtain
  \begin{equation*}
    \lesssim
    C_{A}\cdot
    \||D|^{\frac12}\bra{x}^{-\frac12-\epsilon'}u\|_{L^{2}}
  \end{equation*}
  where we used the assumptions on $A$, provided 
  $\epsilon,\epsilon'$ are small enough. Finally, writing
  \begin{equation*}
    |D|^{\frac12}\bra{x}^{-\frac12-\epsilon'}=
    |D|^{\frac12}\bra{x}^{-\frac12-\epsilon'}
    |D|^{-\frac12}\bra{x}^{\frac12+\epsilon''}
    \cdot
    \bra{x}^{-\frac12-\epsilon''}|D|^{\frac12}
  \end{equation*}
  and using again the boundedness of \eqref{eq:L2bdd}
  and the smoothing estimate \eqref{eq:waveH},
  we get
  \begin{equation*}
    \lesssim
    C_{A}\cdot
    \|\bra{x}^{-\frac12-\epsilon''}|D|^{\frac12}u\|_{L^{2}L^{2}}
    \lesssim
    \|H^{\frac14}f\|_{L^{2}}
    \lesssim
    \|f\|_{\dot H^{\frac12}}.
  \end{equation*}

  For the term $II$, proceeding as before we have
  \begin{equation*}
    \||D|^{\frac1q-\frac1p}II\|_{L^{p}L^{q}}
    \lesssim
    \|\bra{x}^{\frac12+\epsilon}|D|^{-\frac12}
      (A \cdot\nabla u)\|_{L^{2}L^{2}}.
  \end{equation*}
  We use the decomposition
  \begin{equation*}
    \bra{x}^{\frac12+\epsilon}|D|^{-\frac12}
          A \cdot\nabla
    =
    T_{1}ST_{2}
    \cdot \bra{x}^{-\frac12-\epsilon''}|D|^{\frac12}
  \end{equation*}
  where, for some $0<\epsilon<\epsilon''<\epsilon'$,
  \begin{equation*}
    T_{1}=
    \bra{x}^{\frac12+\epsilon}|D|^{-\frac12}
    \bra{x}^{-\frac12-\epsilon'}|D|^{\frac12},
  \end{equation*}
  \begin{equation*}
    S=
    |D|^{-\frac12}\bra{x}^{\frac12+\epsilon'}
    A
    \bra{x}^{\frac12+\epsilon'}|D|^{\frac12}
    \equiv
    |D|^{-\frac12}\bra{x}^{1+\epsilon'+\epsilon''}A|D|^{\frac12}
  \end{equation*}
  and
  \begin{equation*}
    T_{2}=
    |D|^{-\frac12}\bra{x}^{-\frac12-\epsilon'}
    \nabla|D|^{-\frac12}\bra{x}^{\frac12+\epsilon''}.
  \end{equation*}
  The operator $T_{1}$ is $L^{2}$ bounded by \eqref{eq:L2bdd};
  this is true also for $T_{2}$ since
  \begin{equation*}
    (T_{2})^{*}=
    \bra{x}^{\frac12+\epsilon''}\nabla|D|^{-\frac12}
    \bra{x}^{-\frac12-\epsilon'}|D|^{-\frac12}
  \end{equation*}
  and we can use the boundedness of the Riesz operator
  $\nabla|D|^{-1}$ in weighted $L^{2}$ and then again 
  \eqref{eq:L2bdd}. Finally, the operator $S$ or rather
  $S^{*}$ is $L^{2}$ bounded again by the Kato-Ponce
  inequality ($\delta=\epsilon'+\epsilon''$)
  \begin{equation*}
    \||D|^{\frac12}\bra{x}^{1+\delta}A|D|^{-\frac12}
      f\|_{L^{2}}
    \lesssim
    \|\bra{x}^{1+\delta}A\|_{\dot H^{\frac12}_{2n}}
    \||D|^{-\frac12}f\|_{L^{\frac{2n}{n-1}}}
    +
    \|\bra{x}^{1+\delta}A\|_{L^{\infty}}
    \|f\|_{L^{2}}
  \end{equation*}
  and the Sobolev embedding
  $\dot H^{\frac12}\subset L^{\frac{2n}{n-1}}$.
  In conclusion we get as above
  \begin{equation*}
    \||D|^{\frac1q-\frac1p}II\|_{L^{p}L^{q}}
    \lesssim
    C_{A}\cdot
    \|\bra{x}^{-\frac12-\epsilon''}|D|^{\frac12}u\|_{L^{2}L^{2}}
    \lesssim
    \|H^{\frac14}f\|_{L^{2}}
    \lesssim
    \|f\|_{\dot H^{\frac12}}.
  \end{equation*}

  For the last term $III$ the computations are similar:
  we have
  \begin{equation*}
    \||D|^{\frac1q-\frac1p}III\|_{L^{p}L^{q}}
    \lesssim
    \|\bra{x}^{\frac12+\epsilon}|D|^{-\frac12}
      (Vu)\|_{L^{2}L^{2}},
  \end{equation*}
  then by the $L^{2}$ boundedness of
  \begin{equation*}
    \bra{x}^{\frac12+\epsilon}|D|^{-\frac12}
    \bra{x}^{-\frac12-\epsilon'}|D|^{\frac12}
  \end{equation*}
  for $\epsilon'>\epsilon$ we arrive at
  \begin{equation*}
    \lesssim
    \||D|^{-\frac12}\bra{x}^{\frac12+\epsilon'}Vu\|_{L^{2}L^{2}}.
  \end{equation*}
  Next we use Sobolev embedding, the assumption 
  $|V|\lesssim \bra{x}^{-1-\epsilon_{0}}$ and
  H\"{o}lder's inequality
  \begin{equation*}
    \lesssim
    \|\bra{x}^{\frac12+\epsilon'}Vu\|_{L^{2}L^{\frac{2n}{n-1}}}
    \lesssim
    \|\bra{x}^{-\frac12-\epsilon''}\|_{L^{2n}}
    \|\bra{x}^{-1-\epsilon'''} u\|_{L^{2}L^{2}}
  \end{equation*}
  with $\epsilon'+\epsilon''+\epsilon'''=\epsilon_{0}$,
  and recalling \eqref{eq:waveH} we conclude the proof
  of the first estimate in \eqref{eq:wavestr}.

  The proof of the second estimate \eqref{eq:wavestr}
  is identical: just notice that the function
  $u=\sin(t \sqrt{H})H^{-\frac12}g$ solves the wave equation
  \begin{equation*}
    u_{tt}+Hu=0,\qquad
    u(0,x)=0,\quad
    u_{t}(0,x)=g
  \end{equation*}
  and proceed as above.
\end{proof}

\end{document}